\documentclass[12pt]{amsart}

\usepackage{latexsym}
\usepackage{amssymb}
\usepackage{amsfonts}
\usepackage[colorlinks]{hyperref}

\newtheorem{theorem}{Theorem}
\newtheorem{lemma}[theorem]{Lemma}

\theoremstyle{definition}

\theoremstyle{remark}
\newtheorem*{remark}{Remark}

\numberwithin{equation}{section}

\newcommand{\SL}{{\mathrm {SL}}}
\newcommand{\PSL}{{\mathrm {PSL}}}

\newcommand{\Irr}{{\mathrm {Irr}}}

\newcommand{\Stab}{{\mathrm {Stab}}}

\newcommand{\bZ}{\mathbf{Z}}
\newcommand{\bC}{\mathbf{C}}

\newcommand{\bF}{\mathbf{F}}

\newcommand{\Al}{\textup{\textsf{A}}}
\newcommand{\Sy}{\textup{\textsf{S}}}

\def\ttD#1{\mathord{{}^3\mathrm{D}_4(#1)}}
\def\tG#1{\mathord{{}^2\mathrm{G}_2(#1)}}
\def\tB#1{\mathord{{}^2\mathrm{B}_2(#1)}}

\def\psl#1#2{\mathord{\mathrm{PSL}_{#1}(#2)}}
\def\psu#1#2{\mathord{\mathrm{PSU}_{#1}(#2)}}

\begin{document}

\title[Character degree sums of finite groups]
{Character degree sums of finite groups}

\author{Attila Mar\'oti}
\address{Alfr\'{e}d R\'{e}nyi Institute of Mathematics, R\'{e}altanoda utca 13-15, H-1053,
Budapest, Hungary} \email{maroti.attila@renyi.mta.hu}

\author{Hung Ngoc Nguyen}
\address{Department of Mathematics, The University of Akron, Akron,
Ohio 44325, USA} \email{hungnguyen@uakron.edu}

\thanks{The research of the first author was supported by a Marie Curie
International Reintegration Grant within the 7th European
Community Framework Programme, by the J\'anos Bolyai Research
Scholarship of the Hungarian Academy of Sciences, and by OTKA
K84233.}



\date{\today}

\begin{abstract} We present some results on
character degree sums in connection with several important
characteristics of finite groups such as $p$-solvability,
solvability, supersolvability, and nilpotency. Some of them
strengthen known results in the literature.
\end{abstract}

\maketitle


\section{Introduction}

For a finite group $G$, let $\Irr(G)$ denote the set of irreducible
complex characters of $G$ and $T(G)$ the sum of the degrees of these
characters. That is, \[T(G):=\sum_{\chi\in\Irr(G)}\chi(1).\]
Character degree sums of finite groups have been studied extensively
by many authors. For example, Mann \cite{Mann} has shown that
$T(G)/|G|$ is bounded from below if and only if there exist normal
subgroups $N \leq M$ of $G$ so that $|N|$ and $|G/M|$ are bounded
and $M/N$ is abelian. Chapter~11 of the monograph \emph{`Characters
of Finite Groups'}~\cite{Berkovic-Zhmud} is devoted entirely to the
study of character degree sums and consists of several up-to-date
results. One of the highlighted theorems there is due to Berkovich
and Nekrasov -- it classifies all finite groups $G$ with
$T(G)/|G|>1/p$ where $p$ is the smallest prime divisor of $|G|$ such
that a Sylow $p$-subgroup of $G$ is not central. Note that all such
groups are solvable even when $p=2$.

Character degree sums provide a lot of information on the structure
of finite groups. For instance, it has been proved recently by
Isaacs, Loukaki, and Moret\'{o}~\cite{Isaacs-Loukaki-Moreto} and
Tong-Viet~\cite{Tong-Viet} that a finite group $G$ must be solvable
if $T(G)\leq 3 k(G)$ or $T(G)>(4/15) |G|$ where $k(G)$ denotes the
number of conjugacy classes of $G$. In~\cite{Isaacs-Loukaki-Moreto},
it was also proved that if $T(G)<(3/2)k(G)$ or $T(G)<(4/3)k(G)$,
then $G$ is respectively supersolvable or nilpotent.
In~\cite{Barry-MacHale}, Barry, MacHale, and N\'{\i} Sh\'{e}, by
using the classification of Berkovich and Nekrasov mentioned above,
proved that if $T(G)>(1/2)|G|$, then $G$ is supersolvable (more
precisely is nilpotent or has an abelian normal subgroup of index
$2$). In general, when the character degree sum of a finite group
$G$ is small in comparison with the class number of $G$ or is large
in comparison with the order of $G$, it is expected that $G$ is
close to abelian.

In this paper, we obtain more results in this direction. For
convenience, we will call the quantity $T(G)/|G|$ the
\emph{character degree sum ratio} of $G$. Our first result provides
a sufficient criterion for $p$-solvability of finite groups in terms
of the character degree sum ratio.

\begin{theorem}\label{main theorem p-solvable}
Let $G$ be a finite group and $p$ a prime. If $k(G)\geq(3/p^2)|G|$
then $G$ is $p$-solvable. Consequently, if $T(G) \geq
(\sqrt{3}/p)|G|$, then $G$ is $p$-solvable.
\end{theorem}

Next, we extend the main result of~\cite{Tong-Viet} on the relation
between character degree sums and solvability in finite groups. The
following may be compared with \cite[Theorem
11]{Guralnick-Robinson}.

\begin{theorem}\label{main theorem solvable} Let $G$ be a finite group. If $T(G)>(1/4)|G|$ then $G$ is solvable of Fitting height at most $4$
or $G = \Al_{5} \times Z$ for some abelian group $Z$.
\end{theorem}

Theorem~\ref{main theorem solvable} shows that groups isoclinic to
$\Al_5$ have character degree sum ratio substantially larger than
other non-solvable groups. Two groups are said to be
\emph{isoclinic} if there are isomorphisms between their inner
automorphism groups and their derived subgroups such that the
isomorphisms are compatible with the commutator map. Any group
isoclinic to a simple group is isomorphic to a direct product of the
simple group with an abelian group but this is not true for
arbitrary groups. Since isoclinic groups have same proportions of
degrees of irreducible complex representations, we observe that the
character degree sum ratio is invariant under isoclinism, see
Theorem~\ref{T(G)/|G| is invariant under isoclinism }. It is known
that if $T(G)>(2/3)|G|$ then $G$ is nilpotent
(see~\cite[Chapter~11]{Berkovic-Zhmud}), and that the bound here
cannot be improved as $T(\Sy_3)=(2/3)|\Sy_3|$. Similar to
Theorem~\ref{main theorem solvable}, the following result shows that
the character degree sum ratio of $\Sy_3$ is substantially larger
than that of other non-nilpotent groups not isoclinic to $\Sy_3$.

\begin{theorem}\label{main theorem nilpotent} Let $G$ be a finite group. If $T(G)> (\sqrt{3/8})|G|$, then
$G$ is either abelian, isoclinic to a $2$-group, to a $3$-group, to
$\Sy_3$, or to $D_{10}$.
\end{theorem}

As mentioned already, it was proved in~\cite{Barry-MacHale} that if
$T(G)>(1/2)|G|$ then $G$ is supersolvable by using the long and
complicated classification of finite groups $G$ with $T(G)/|G|>1/p$
where $p$ is the smallest prime divisor of $|G|$ such that a Sylow
$p$-subgroup of $G$ is not central. We present here a short proof of
this fact that is independent and indeed fundamentally different
from the classification of Berkovich and Nekrasov.

\begin{theorem}\label{main theorem supersolvable} Let $G$ be a finite group. If $T(G)>(1/2)|G|$, then $G$
is supersolvable.
\end{theorem}

We note that the bound in this theorem cannot be improved since
$T(\Al_4)=(1/2)|G|$. It would be interesting if one can show that
groups isoclinic to $\Al_4$ have significantly larger character
degree sum ratio than other non-supersolvable groups.

To end this introduction, we remark that our proofs of
Theorems~\ref{main theorem p-solvable} and~\ref{main theorem
solvable}, which are carried out respectively in
Sections~\ref{section p-solvable} and~\ref{section solvable}, rely
on the classification of finite simple groups. On the other hand,
proofs of Theorems~\ref{main theorem nilpotent} and~\ref{main
theorem supersolvable} in Sections~\ref{section nilpotent}
and~\ref{section supersovable} are classification-free.


\section{Preliminaries}

There are two well-known bounds for $T(G)$, the sum of the complex
irreducible character degrees of a finite group $G$. On one hand, by
a formula of Frobenius and Schur (see~\cite[Corollary~4.6]{Isaacs}),
$T(G)$ can be bounded from below by the number $I(G)$ of elements of
$G$ of orders dividing $2$. On the other hand, using the
Cauchy-Schwarz inequality, we have $T(G) \leq \sqrt{k(G) |G|}$ where
we recall that $k(G)$ is the number of complex irreducible
characters of $G$. By introducing the notations $t(G) = T(G)/|G|$,
$i(G) = I(G)/|G|$, and $d(G) = k(G)/|G|$, these inequalities can be
stated as follows.

\begin{lemma}
\label{l00} For a finite group $G$ we have $i(G) \leq t(G) \leq
\sqrt{d(G)}$.
\end{lemma}

The invariant $d(G)$ is indeed the probability that a randomly
chosen pair of elements of $G$ commute. That is,
\[d(G)=\frac{1}{|G|^2}|\{(x,y)\in G\times G\mid xy=yx\}|.\]
This quantity is often referred to as the \emph{commuting
probability} (or \emph{commutativity degree}) of $G$. The study of
the commuting probability of finite groups dates back at least to
work of Gustafson in the seventies. One of the earliest results is
the following.

\begin{lemma}[Gustafson \cite{Gustafson}]
\label{Gustafson} If $G$ is a non-abelian finite group then $d(G)
\leq 5/8$.
\end{lemma}

In 1962, Nagao \cite{Nagao} showed that for a normal subgroup $N$ of
a finite group $G$ we have $k(G) \leq k(N) k(G/N)$. This implies the
following useful result.

\begin{lemma}
\label{Nagao} For a normal subgroup $N$ of a finite group $G$ we
have $d(G) \leq d(N) d(G/N)$.
\end{lemma}

One of the deepest results on the commuting probability of a finite
group is due to Guralnick and Robinson.

\begin{lemma}[Guralnick and Robinson~\cite{Guralnick-Robinson}]
\label{l1} Let $\bF(G)$ be the Fitting subgroup of a finite group
$G$. Then $d(G) \leq {|G:\bF(G)|}^{-1/2}$.
\end{lemma}

Two remarks are in order. First, the proof of Lemma~\ref{l1}
in~\cite{Guralnick-Robinson} depends on the classification. As our
proof of Theorem~\ref{main theorem solvable} uses Lemma~\ref{l1} in
both solvable and non-solvable cases, it depends on the
classification as well. Second, Neumann~\cite{Neumann} has shown
that if $d(G)$ is bounded from below by some real positive number
$r$ then $G$ contains a normal subgroup $H$ so that $|G:H|$ and $H'$
are bounded by some function of $r$. Lemma~\ref{l1} implies that
this subgroup $H$ can be taken to be nilpotent.

Finally, the following important result will also be used.

\begin{lemma}[Gallagher \cite{Ga}]
\label{l4} Let $G$ be a finite group, $N$ be a normal subgroup in
$G$, $\chi$ be an irreducible character of $N$, and $I(\chi)$ be
its inertia subgroup. Then the number of irreducible characters of
$G$ which lie over ($G$-conjugates of) $\chi$ is at most
$k(I(\chi)/N)$.
\end{lemma}


\section{$p$-solvability}\label{section p-solvable}

This section is devoted to proving Theorem~\ref{main theorem
p-solvable}. We need two lemmas.

\begin{lemma}
\label{l88} If $G$ is a non-abelian finite simple group whose
order is divisible by a prime $p$ then $p^{2}/3 < |G|^{1/2}$
unless possibly if $G$ is $\psl{2}{q}$, $\psl{3}{q}$,
$\psu{3}{q}$, $\tB{q}$, $\tG{q}$, or $\ttD{q}$ for some prime
power $q$.
\end{lemma}

\begin{proof}
This is elementary computation using the list of orders of finite
simple groups found in~\cite[pages 170-171]{KL}, say.
\end{proof}

\begin{lemma}
\label{l89} Let $p$ be a prime divisor of the order of a non-abelian
finite simple group $G$. Then we have the following:
\begin{enumerate}
\item Let $G = \psl{2}{q}$. If $q \geq 4$ is even then $d(G) =
1/((q-1)q) < 3/p^{2}$. If $q$ is odd then $d(G) = (q+5)/((q^{2}-1)q)
< 3/p^{2}$.

\item Let $G = \psl{3}{q}$. If $q \geq 4$ then $$d(G) \leq
(q^{2}+3q)/((1/3)q^{3}(q^{2}-1)(q^{3}-1)) < 3/p^{2}.$$ Furthermore
$d(\psl{3}{3}) = 1/468 < 3/169$.

\item Let $G = \psu{3}{q}$. Then $d(G) \leq
(q^{2}+q+2)/((1/d)q^{3}(q^{2}-1)(q^{3}+1)) < 3/p^{2}$ where $d$ is
$3$ if $3$ divides $q+1$ and is $1$ otherwise.

\item Let $G = \tB{q}$. Then $d(G) =
(q+3)/(q^{2}(q^{2}+1)(q-1)) < 3/p^{2}$.

\item Let $G = \tG{q}$. Then $d(G) =
(q+8)/(q^{3}(q^{3}+1)(q-1)) < 3/p^{2}$.

\item Let $G = \ttD{q}$. Then $$d(G) =
(q^{4}+q^{3}+q^{2}+q+6)/(q^{12}(q^{8} +
q^{4}+1)(q^{6}-1)(q^{2}-1)) < 3/p^{2}.$$
\end{enumerate}
\end{lemma}

\begin{proof}
Note that $p$ can be chosen to be the largest prime divisor of
$|G|$. (By Burnside's theorem $p$ is at least $5$.) From the formula
for $|G|$, which can be found in~\cite[pages 170-171]{KL}, we can
derive a `good' upper bound for $p$ in terms of $q$. For such it is
worthy to note that if $2^{k}+1$ is a Fermat prime then $k$ is a
power of $2$ and similarly if $3^{k}+1$ is a prime then $k$ is a
power of $2$. Upper bounds or the exact values for $k(G)$ can be
found in~\cite[pages 3049-3050]{FG} and in~\cite{Mac}.

We will prove (4) as a demonstration. Let $G=\tB{q}$ with
$q=2^{2n+1}$ for some integer $n\geq 1$. Note that
$|G|=q^2(q^2+1)(q-1)$. Now $q^2+1=2^{4n+2}+1$ cannot be a prime
since $4n+2$ is not a power of $2$. We deduce that $p\leq(q^2+1)/3$
if $p$ is a prime divisor of $q^2(q^2+1)(q-1)$. Using the fact that
$k(G)=q+3$, it is now straightforward to check that $d(G)<3/p^2$.
\end{proof}

We are now in the position to prove Theorem~\ref{main theorem
p-solvable}.

\begin{proof}[Proof of Theorem~\ref{main theorem p-solvable}] Let $G$ be a finite group with $k(G)\geq(3/p^2)|G|$. Then we have $d(G)\geq 3/p^2$. By
Lemma~\ref{Nagao} we see that $d(G) \leq d(C)$ for any composition
factor $C$ of $G$. Therefore, to prove the first part of the
theorem, it is sufficient to show $d(S) < 3/p^{2}$ for every
non-abelian simple group $S$ whose order is divisible by $p$. By
Lemma~\ref{l1}, we have $d(S) = k(S)/|S| \leq {|G|}^{-1/2}$. This
and Lemma~\ref{l88} imply that $d(S) < 3/p^{2}$ unless $S$ is
isomorphic to one of the groups treated in Lemma~\ref{l89}. In all
these exceptional cases we have $d(S) < 3/p^{2}$ by
Lemma~\ref{l89}.

The second part of the theorem follows from the first part and
Lemma~\ref{l00}.\end{proof}


\section{Solvability}\label{section solvable}

In this section we will prove Theorem~\ref{main theorem solvable}.

\begin{proof}[Proof of Theorem~\ref{main theorem solvable}] Let $G$ be a finite group with $t(G) > 1/4$. Then, by Lemma
\ref{l00}, we have that $d(G) > 1/16$. By using Lemma~\ref{l1}, we
obtain that the index of the Fitting subgroup $\bF(G)$ in $G$ is
less than $256$. Hence if $G$ is solvable then it has Fitting height
at most $4$ by Gap~\cite{GAP}. So we may assume that $G$ is
non-solvable.

Let $S$ be the largest solvable normal subgroup of $G$. Clearly, $S$
has index less than $256$ in $G$ by Lemma~\ref{l1}. Hence $G/S$ is
isomorphic to $\Al_5$, $\Sy_5$, or $\PSL(2,7)$.

Suppose first that $S$ is non-abelian. Then we have $d(S) \leq
5/8$ by Lemma \ref{Gustafson}. We also have $d(G/S) \leq 1/12$.
So, by Lemma \ref{Nagao}, we have $$1/16 < d(G) \leq d(S) d(G/S)
\leq (5/8)(1/12),$$ which is a contradiction. We conclude that $S$
is abelian.

We may assume that $G/S$ is isomorphic to $\Al_5$. For otherwise
$$1/16 < d(G) \leq d(G/S) \leq 7/120$$ which is impossible.

The factor group $G/S \cong \Al_5$ acts naturally on
$\mathrm{Irr}(S)$. Each orbit has size $1$ or at least $5$. Let $r$
be the number of orbits of length $1$. Since every subgroup of $A_5$
has at most $5$ conjugacy classes, we have $k(G) \leq 5r + (|S|-r)$
by Clifford's theorem and Lemma~\ref{l4}. Thus we have \[1/16 < d(G)
\leq \frac{|S|+4r}{60|S|}\] which forces $r > (11/16)|S|$. Since
more than half of the character group $\mathrm{Irr(S)}$ is fixed by
$G/S$, we must have that $G/S$ acts trivially on $\mathrm{Irr}(S)$.
But then, by Brauer's permutation lemma, $G/S$ must act trivially on
$S$ as well, which means that $Z = \bZ(G) = S$.

Let $H$ be the last term in the derived series of $G$. There exists
a solvable normal subgroup $T$ in $H$ with $H/T \cong \Al_5$. Since
$TZ/Z$ is a solvable normal subgroup in $G/Z \cong \Al_5$ we must
have $T \leq Z$. So $H$ is perfect and a central extension of
$\Al_5$. This means that $H$ is either $\Al_5$ or $\SL(2,5)$. In the
former case we have $G = \Al_{5} \times Z$, so assume that $H \cong
\SL(2,5)$. We conclude that $G$ is a central product of the normal
subgroups $H$ and $Z$.

Let the intersection of $H$ and $Z$ be $D = \langle a \rangle$.
This is a central subgroup of order $2$. Put $\overline{G} = H
\times Z$. By~\cite[Lemma 5.2]{K}, the sum $s(\overline{G})$ of
the degrees of those complex irreducible characters of
$\overline{G}$ which have $\langle (a,a) \rangle$ in their kernel
is at least $T(G)$. By the character table of $\SL(2,5)$ it is
easy to see that
$$T(G) \leq s(\overline{G}) = T(\SL(2,5)) \cdot |Z|/2 = 15 |Z|.$$ We conclude that
$$\frac{1}{4} < \frac{T(G)}{|G|} \leq \frac{15|Z|}{60|Z|} = \frac{1}{4},$$ which is
a contradiction. This completes the proof.\end{proof}



\section{Nilpotency}\label{section nilpotent}

The aim of this section is to prove Theorem~\ref{main theorem
nilpotent}. To do that, we need to recall some basic facts on
\emph{isoclinism}.

Two groups $G$ and $H$ are said to be \emph{isoclinic} if there are
isomorphisms $\varphi:G/\bZ(G)\rightarrow H/\bZ(H)$ and $\phi:
G'\rightarrow H'$
such that \begin{align*} \text{ if } \varphi(g_1\bZ(G))&=h_1\bZ(H)\\
\text{ and } \varphi(g_2\bZ(G))&=h_2\bZ(H),\\ \text{ then }
\phi([g_1,g_2])&=[h_1,h_2].\end{align*} This concept was introduced
by Hall in~\cite{Hall} as a structurally motivated classification
for finite groups, especially for $p$-groups. It is well-known that
several characteristics of finite groups such as nilpotency,
supersolvability, or solvability are invariant under isoclinism,
see~\cite{Bioch-Waall}. We will see that the quantity $T(G)/|G|$ is
also invariant under isoclinism.

Isoclinic groups have the same proportions of degrees of irreducible
complex representations. We are aware that this result is known but
we could not track down a formal reference. We refer the reader to
the groupwiki webpage~\cite{webpage} for a proof.

\begin{lemma}\label{lemma propotions of degrees of irreducible characters} Let $G$
and $H$ be isoclinic finite groups and let $d$ be a positive
integer. Suppose that $G$ and $H$ have respectively $m$ and $n$
irreducible characters of degree $d$. Then $m$ is nonzero if and
only if $n$ is nonzero. In that case, we have
\[\frac{m}{n}=\frac{|G|}{|H|}.\]
\end{lemma}

The next two results are crucial in the proof of Theorems~\ref{main
theorem nilpotent}.

\begin{theorem}\label{T(G)/|G| is invariant under isoclinism } Let $G$
and $H$ be isoclinic finite groups. Then $T(G)/|G|=T(H)/|H|$.
\end{theorem}

\begin{proof} By Lemma~\ref{lemma propotions of degrees of irreducible
characters}, we know that $G$ and $H$ have the same irreducible
character degrees. So we assume that $d_1,d_2,...,d_k$ are all
character degrees of $G$ and $H$. Let $m_{i}$ and $n_{i}$ denote the
multiplicity of the degree $d_i$ of $G$ and $H$, respectively. We
have, by Lemma~\ref{lemma propotions of degrees of irreducible
characters},
$$T(G)=\sum_{i=1}^k m_id_i=\sum_{i=1}^k
n_i\frac{|G|}{|H|}d_i=\frac{|G|}{|H|}T(H),$$ and the theorem
follows.
\end{proof}

From this proof it also follows that $d(G) = d(H)$ whenever $G$ and
$H$ are isoclinic finite groups, a fact proved earlier by Lescot
\cite{Lescot}.

A \emph{stem group} is defined to be a group whose center is
contained inside its derived subgroup. It is known that every group
is isoclinic to a stem group and if we restrict to finite groups, a
stem group has the minimum order among all groups isoclinic to it,
see~\cite{Hall} for more details.

\begin{lemma}\label{isoclinic lemma} For every finite group $G$,
there is a finite group $H$ isoclinic to $G$ such that $|H|\leq |G|$
and $\bZ(H)\subseteq H'$.
\end{lemma}

\begin{proof}[Proof of Theorem~\ref{main theorem nilpotent}] Let $G$ be a finite group with $T(G)>\sqrt{3/8}|G|$. Then we have $d(G) > 3/8$
by Lemma~\ref{l00}. Using the table in \cite[Page 246]{Rusin}, we
deduce that one of the following cases holds.
\begin{enumerate}
\item Both $G/\bZ(G)$ and $G'$ are $2$-groups;

\item Both $G/\bZ(G)$ and $G'$ are $3$-groups;

\item $G/\bZ(G) \cong \Sy_3$ and $|G'| = 3$;

\item $G/\bZ(G) \cong D_{10}$ and $|G'| = 5$.
\end{enumerate}

By Lemma \ref{isoclinic lemma}, we may assume that $\bZ(G) \leq G'$.
Thus $G$ is a $2$-group in case (1) and is a $3$-group in case (2).
In cases (3) and (4) the center of $G$ cannot coincide with $G'$ so
$\bZ(G) = 1$. This means that $G$ must be isomorphic to $\Sy_3$ and
to $D_{10}$ in the respective cases. The proof is now
complete.\end{proof}


\section{Supersolvability}\label{section supersovable}

In this section we will prove Theorem \ref{main theorem
supersolvable}. We first recall a well-known lemma, which can be
found in~\cite[VI.8.6]{Huppert}.

\begin{lemma}\label{if N<Frattini(G) and G/N is supersolvable} Let
$N$ be a normal subgroup of $G$ that is contained in the Frattini
subgroup of $G$. If $G/N$ is supersolvable, then $G$ is
supersolvable.
\end{lemma}

The next two lemmas are crucial in the proof of Theorem~\ref{main
theorem supersolvable}.

\begin{lemma}\label{T(G)>1/2|G|, then T(G/N)>1/2|G/N|} Let $N$ be a
normal subgroup of $G$ that is contained in $G'$. If
$T(G)>(1/2)|G|$, then $T(G/N)>(1/2)|G/N|$.
\end{lemma}

\begin{proof} We have
\begin{align*}T(G)>\frac{1}{2}|G|&\Leftrightarrow \sum_{\chi\in\Irr(G)}\chi(1)>\frac{1}{2}\sum_{\chi\in\Irr(G)}\chi(1)^2\\
&\Leftrightarrow \sum_{\chi\in\Irr(G)}(\chi(1)^2-2\chi(1))<0\\
&\Leftrightarrow \sum_{\chi\in\Irr(G),\chi(1)\geq3}(\chi(1)^2-2\chi(1))<\sum_{\chi\in\Irr(G),\chi(1)=1}(2\chi(1)-\chi(1)^2)\\
&\Leftrightarrow
\sum_{\chi\in\Irr(G),\chi(1)\geq3}(\chi(1)^2-2\chi(1))<
|\{\chi\in\Irr(G),\chi(1)=1\}|\\ &\Leftrightarrow
\sum_{\chi\in\Irr(G),\chi(1)\geq3}(\chi(1)^2-2\chi(1))< |G:G'|.
\end{align*}
Since $N\subseteq G'$, we see that
$[(G/N):(G/N)']=[(G/N):(G'/N)]=[G:G']$. It follows, as every
irreducible character of $G/N$ can be considered as an irreducible
character of $G$, that
\[\sum_{\chi\in\Irr(G/N),\chi(1)\geq3}(\chi(1)^2-2\chi(1))<
|(G/N):(G/N)'|,\] which implies that $T(G/N)>(1/2)|G/N|$, as
desired.
\end{proof}

The following lemma is a consequence of a result of Aschbacher and
Guralnick~\cite{AG}.

\begin{lemma}\label{non-coprime k(GV) problem} Let $p$ be a prime
and let $V$ be a finite dimensional vector space over the field of
$p$ elements. Let $G$ be a group acting faithfully and irreducibly
on $V$. Then $|G:G'|<|V|$.
\end{lemma}

\begin{proof} If $N$ is a non-trivial normal subgroup of G then by Clifford's
theorem $V$ is a completely reducible $N$-module. Therefore, if
$N$ is furthermore a $p$-subgroup then the only irreducible
$N$-submodule of $V$ is the trivial module. So $V$ must be trivial
as an $N$-module. We conclude that $O_p(G)$, the maximal normal
$p$-subgroup of $G$, is trivial. Now the lemma follows
by~\cite[Theorem 3]{AG}.
\end{proof}

\begin{remark} The proof of Aschbacher and Guralnick~\cite[Theorem 3]{AG}
requires the finite simple group classification. However, when the
group $G$ is solvable, their proof actually does not depend on the
classification. Our proof of Theorem~\ref{main theorem
supersolvable} below uses Lemma~\ref{non-coprime k(GV) problem} only
in the case where $G$ is solvable and therefore it neither does not
depend on the classification. We thank the referees for bringing
this discussion to our attention.
\end{remark}

We are now ready to prove Theorem \ref{main theorem supersolvable}.

\begin{proof}[Proof of Theorem~\ref{main theorem supersolvable}] Assume, to the contrary, that Theorem~\ref{main theorem
supersolvable} is not true and let $G$ be a minimal counterexample.
In particular, we have $T(G)>(1/2)|G|$ and therefore $G$ is solvable
by~\cite[Theorem~A]{Tong-Viet}. If $G'=1$ then $G$ is abelian and we
are done. So we can assume that $G'$ is nontrivial. Let $N$ be a
minimal normal subgroup of $G$ with $N\subseteq G'$, then $N$ must
be elementary abelian by the solvability of $G$. Also, using
Lemma~\ref{T(G)>1/2|G|, then T(G/N)>1/2|G/N|}, we deduce that
$T(G/N)>(1/2)|G/N|$, which implies that $G/N$ is supersolvable by
the minimality of $G$.

Again as $G$ is not supersolvable, Lemma~\ref{if N<Frattini(G) and
G/N is supersolvable} implies that $N$ is not contained in the
Frattini subgroup of $G$. Hence, $N$ is not contained in a maximal
subgroup $M$ of $G$ so that $NM=G$. Since $N$ is abelian, we see
that $N\cap M\lhd G$. Now the minimality of $N$ and the fact that
$N$ is not contained in $M$ imply that $N\cap M=1$. Equivalently,
\[G\cong N\rtimes M.\]

If $N\subseteq \bZ(G)$ then $N$ would be cyclic and hence $G$ is
supersolvable. Therefore, we assume that $N\nsubseteq \bZ(G)$ or
equivalently $[N,M]>1$. It is clear that $[N,M]$ is normal in $G$.
The minimality of $N$ then implies that $[N,M]=N$. Therefore, no
non-principal character of $N$ is fixed under the conjugation action
of $M$.

If there is a linear character $\alpha$ of $N$ that is in an
$M$-orbit of size 2, then the stabilizer $\Stab_M(\alpha)$ of
$\alpha$ in $M$ is a normal subgroup of $M$ of index $2$. The
conjugation action of $M/\Stab_M(\alpha)$ on $\Irr(N)$ is
irreducible and has no nontrivial fixed point. Therefore $|N|$ is
a prime so that $N$ is cyclic. This would imply that $G$ is
supersolvable since $G/N$ is supersolvable, a contradiction.

From now on we can assume that every nontrivial orbit of the action
of $M$ on $\Irr(N)$ has size at least $3$. It follows that every
nontrivial orbit of the action of $G$ on $\Irr(N)$ has size at least
$3$. Now consider the group $C = \bC_{M}(N)$. This is normal in $M$
and centralizes $N$, so it is normal in $G = MN$. Hence $K = N
\times C$ is a normal subgroup of $G$. The subset $$S = \Irr(K)
\setminus \{ 1 \otimes \chi : \chi \in \Irr(C) \}$$ of $\Irr(K)$ is
$G$-invariant and every $G$-orbit has size at least $3$. By Clifford
theory, each $G$-orbit of size $d$ in $S$ produces at least one
irreducible character of $G$ of degree divisible by $d$ and
different $G$-orbits in $S$ produce different characters of $G$.
Thus we have that
\[\sum_{d\geq 3}d\cdot n_d(G)\geq (|N|-1) \cdot T(C),\]
where $n_d(G)$ denotes the number of irreducible complex
characters of $G$ of degree $d$. Since $d\leq d^2-2d$ for every
$d\geq 3$, it follows that
\[\sum_{d\geq 3}(d^2-2d)\cdot n_d(G)\geq (|N|-1) \cdot T(C).\] Equivalently,
\[\sum_{\chi\in\Irr(G),\chi(1)\geq3}(\chi(1)^2-2\chi(1))\geq (|N|-1) \cdot T(C).\]

Recall that $N$ is elementary abelian. Therefore, we can consider
$N$ as a finite dimensional vector space over a field of $p$
elements for some prime $p$. Also, since $N$ is a minimal normal
subgroup of $G$, the conjugation action of $G$ on $N$ is
irreducible. In particular, the factor group $M/C$ can be
considered as a group of linear transformations acting faithfully
and irreducibly on $N$. Since the group $M'C/C$ is normal in $M/C$
and the quotient is abelian, we see by Lemma~\ref{non-coprime
k(GV) problem} that $|N|-1 \geq |M|/|M'C|$.

We claim that $(|M|/|M'C|) \cdot T(C) \geq |M:M'| = |G:G'|$. This
would be sufficient for our purposes since this would give
\[\sum_{\chi\in\Irr(G),\chi(1)\geq3}(\chi(1)^2-2\chi(1))\geq
|G:G'|,\] and as we have already seen in the proof of
Lemma~\ref{T(G)>1/2|G|, then T(G/N)>1/2|G/N|}, this inequality is
equivalent to \[T(G)\leq \frac{1}{2}|G|,\] which violates the
hypothesis.

To prove the claim it is sufficient to prove the inequality since
the equality follows from $N \leq G'$. For that it is sufficient to
verify $$T(C) \geq |M'C|/|M'| = |M'C/M'| = |C:(M' \cap C)|.$$ But
$M' \cap C \geq C'$ and so $T(C) \geq |C:C'| \geq |C:(M' \cap C)|$,
as desired.\end{proof}


\section*{Acknowledgement} We are grateful to the anonymous referees
for excellent comments and insightful suggestions that have
significantly improved the exposition of the paper.


\end{document}